\documentclass[11pt, onesided, leqno]{amsart}
\usepackage{enumerate,enumitem}
\usepackage[all]{xy}
\usepackage{mathtools}
\usepackage{amsmath,amssymb,mathrsfs,geometry,color,bm}
\usepackage{pb-diagram}

\newcommand{\doublespace}
   {\addtolength{\baselineskip}{0.25\baselineskip}}

\newtheorem{thm}{Theorem}[section]

\newtheorem{lem}[thm]{Lemma}

\newtheorem{prop}[thm]{Proposition}

\theoremstyle{definition}

\newtheorem{rem}[thm]{Remark}
\newtheorem{exam}[thm]{Example}

\numberwithin{equation}{section}

\newcommand{\N}{\mathbb{N}}
\newcommand\R{\mathbb{R}}

\newcommand{\FC}{R}     
\newcommand{\tR}{{\tilde R}}
\newcommand{\tS}{{\tilde S}}

\renewcommand\P{\mathcal{P}(\R)}
\newcommand{\PM}{\mathcal{P}(\R_+)}
\newcommand\ID{\mathcal{ID}}

\newcommand{\Aa}{\kappa} %
\newcommand{\Bb}{\tilde\kappa} %

\newcommand\CH{\mathbf{S}}
\newcommand\BH{\mathbf{B}}
\newcommand\FH{\mathbf{F}}
\newcommand\PO{\mathbf{P}}

\newcommand\MP{{\bm \pi}}
\newcommand{\fs}{\mathbf{f}}
\newcommand{\bs}{\mathbf{b}}
\newcommand\cs{\mathbf{s}}
\newcommand\ed{\mathbf{e}}
\newcommand\uni{\mathbf{u}}
\newcommand\WS{\mathbf{w}}

\newcommand{\fmax}{\Box \hspace{-.7em} \lor}

\newcommand{\fmaxxx}{\Box \hspace{-.5em} \lor}

\newcommand{\bmax}{\cup \hspace{-.65em} \lor}

\newcommand{\bmaxxx}{\cup \hspace{-.5em} \lor}

\title[Homomorphisms relative to additive convolutions and max-convolutions]{Homomorphisms relative to additive convolutions and max-convolutions:  free, boolean and classical cases}
\author{Takahiro Hasebe \and Yuki Ueda}

\subjclass[2010]{Primary 46L54; Secondary 60E07; 60G70.}
\keywords{additive convolution, max-convolution, free multiplicative convolution, stable distributions, extreme value distributions, infinitely divisible distributions, Bercovici-Pata bijection}

\address{Takahiro Hasebe:  
Department of Mathematics, Hokkaido University,
Kita 10, Nishi 8, Kita-Ku, Sapporo, Hokkaido, 060-0810, Japan}

\email{thasebe@math.sci.hokudai.ac.jp}

\address{
Yuki Ueda: 
Department of General Science, National Institute of Technology, Ichinoseki College,
Takanashi, Hagisho, Ichinoseki, Iwate 021-8511, Japan }
\email{yuki1114@ichinoseki.ac.jp}

\begin{document}

\maketitle  
\doublespace
\pagestyle{myheadings} 
%

\begin{abstract}
We introduce new homomorphisms relative to additive convolutions and max-convolutions in free, boolean and classical cases. Crucial roles are played by the limit distributions for free multiplicative law of large numbers.   
\end{abstract}

\section{Introduction}

In non-commutative probability theory, self-adjoint operators are interpreted as (real-valued) random variables. A striking feature is that various notions of independence exist for those random variables.  In particular, free independence (also called freeness) has found its applications to operator algebras and random matrices, so that it has been intensively studied; see \cite{NS06,MS17} and references therein.  
 Several convolutions of probability measures are associated with each notion of independence: the additive convolution describes the law of the sum of independent random variables; multiplicative convolution describes the law of product of independent random variables; max-convolution describes the law of maximum (in the Ando's sense \cite{A89}) of independent random variables.

Limit theorems in non-commutative probability are one of the main topics of interests. Limit theorems for addition of independent random variables are quite parallel to those in probability theory as observed in the pioneering work \cite{BP99} and other papers, e.g.\ \cite{CG08,Wan08}. Such a parallelism is remarkable, while the effect of non-commutativity is not very visible in the results. 
On the other hand, the non-commutativity of random variables appears more clearly in multiplicative law of large numbers. 
In probability theory,  the asymptotic behavior of products of a large number of independent positive random variables reduces to the usual law of large numbers for addition of real-valued random variables, because the exponential mapping is a homomorphism from $(\R,+)$ onto $((0,\infty), \cdot)$. However, for non-commutative random variables such as random matrices, law of large numbers for products of independent random variables is no longer obvious because the exponential mapping is not a homomorphism (an attempt to recover the homomorphism property is found in \cite{AA17} in the unitary case). 

In free probability, Tucci \cite{Tuc10}, Haagerup and M\"oller \cite{HM13} formulated and investigated the free multiplicative law of large numbers for non-negative free random variables. In terms of probability measures, it can be formulated as the convergence of  
\begin{equation}\label{eq:FMLLN}
(\mu^{\boxtimes n})^\frac{1}{n}, \qquad n\in \N, 
\end{equation}
where $\mu$ is a probability measure on $\R_+:=[0,\infty)$, $\boxtimes$ is free multiplicative convolution (see Section \ref{sec2}) and $\nu^{\alpha}$ denotes the push-forward of a measure $\nu$ by the mapping $x\mapsto x^\alpha$ for $\alpha \in\R$. The limit distribution of the sequence \eqref{eq:FMLLN} exists and is denoted by $\Phi(\mu)$ below. When $\mu\ne\delta_0$, the limit distribution is characterized by 
\begin{equation}\label{eq:HM}
 \Phi(\mu)(\{0\})=\mu(\{0\})\qquad \text{and}\qquad \Phi(\mu)\left(\left[0,\frac{1}{S_\mu(t-1)}\right]\right) = t, \qquad t \in (\mu(\{0\}),1), 
\end{equation}
where $S_\mu$ is the $S$-transform of $\mu$ (see \cite[Theorem 2]{HM13}). Unlike the additive case where the limit distribution is always a delta measure, the above limit distribution $\Phi(\mu)$ need not be a delta measure and it is far from universal; actually the mapping $\mu\mapsto \Phi(\mu)$ is injective.  This kind of non-universal limits should appear in the level of random matrices of finite size since free probability somehow describes the large size limit of random matrices. However,  as far as the authors know, there have not been such results for random matrices. 

The second-named author discovered in \cite[Theorems 1.1 and 1.2]{Ua} an identity involving $\Phi,$ free additive convolution $\boxplus$ and free max-convolution $\fmax$, boolean additive convolution $\uplus$ and boolean max-convolution $\bmax$ (see Section \ref{sec2} for those convolutions), in the forms   
\begin{equation}\label{eq:Ueda}
\Phi(D_{1/t} (\mu^{\boxplus t})) = \Phi(\mu)^{\fmaxxx t}, \qquad t\ge1 \qquad  \text{and} \qquad \Phi(D_{1/t} (\mu^{\uplus t})) = \Phi(\mu)^{\bmaxxx t}, \qquad t > 0, 
\end{equation}
where $D_c (\nu)$, called the dilation or scaling, is the push-forward of a measure $\nu$ by the mapping $x \mapsto c x$ for $c\ge0$. An interesting point here is that there is no apparent or natural reason that suggests a connection of free multiplicative law of large numbers, free additive convolution and free max-convolution, or boolean ones.   

The present paper seeks for a better understanding of those unexpected connections between various convolutions and $\Phi$. In fact, a conceptual explanation of those connections is still lacking, but in this paper we present further non-trivial identities closely related to \eqref{eq:Ueda}. The main result is that the mapping 
\begin{equation}\label{eq:Theta}
\mu\mapsto\Theta(\mu):= \Phi(\MP^{-1}\boxtimes \mu)
\end{equation}
is a homomorphism from the semigroup $(\PM, \boxplus)$ to the semigroup $(\PM,\fmax)$, where $\PM$ is the set of Borel probability measures on $\R_+$ and $\MP$ is the standard Marchenko-Pastur law. Moreover, the mapping 
\begin{equation}\label{eq:Xi}
\mu\mapsto\Xi(\mu):= \Phi(\MP^{-1}\boxtimes \MP \boxtimes \mu)
\end{equation}
is a homomorphism from $(\PM, \uplus)$ to $(\PM,\bmax)$. 

Section \ref{sec2} summarizes several convolutions and computational tools needed in this paper.  Section \ref{sec3} is the main part, where the homomorphism properties \eqref{eq:Theta} and \eqref{eq:Xi} for additive convolution and max-convolution are established. A classical version of the homomorphism is also introduced. Section \ref{sec4} is devoted to intertwining relations of different types of homomorphisms including Bercovici--Pata bijections and some others related to stable distributions.

\section{Preliminaries}\label{sec2}

We summarize convolutions and the analytic characterizations of them.  Some convolutions and results below can be extended to the set $\P$ of probability measures on $\R$, but we restrict ourselves mostly to $\PM$ to get refined descriptions of transforms. Also, this setting is sufficient for our purpose.  
\subsection{Classical convolutions}

For $\mu \in \PM$, let $C_\mu, M_\mu$ be the cumulant transform and Mellin transform of $\mu \in \PM$, respectively: 
\begin{align}
C_\mu(t) &= \log \int_{\R_+} e^{t u} \,d\mu(u), \qquad t\le0, \\
M_\mu(t) &= \int_{(0,\infty)} x^t \,d\mu(x), \qquad t \in i\R.
\end{align}
The domains of the functions $C_\mu$ and $M_\mu$ may be extended when the integrals converge. Classical additive convolution $\ast$ (see e.g.\ \cite[Definition 2.4]{S13}) and   multiplicative convolution $\circledast$ are respectively characterized by 
\[
C_{\mu\ast\nu} = C_\mu + C_\nu, \qquad M_{\mu\circledast\nu}=M_\mu M_\nu.
\] 
Note that the study of multiplicative convolution reduces to additive convolution by the exponential mapping.

Finally, the max-convolution $\lor$ of $\mu,\nu\in \PM$ is characterized by
\[
(\mu\lor \nu)([0,t])=\mu([0,t])\nu([0,t]), \qquad t\ge0,
\]
see \cite{R87} for further information: extreme value distributions, higher-dimensional max-convolution and max-infinitely divisible distributions, etc.

\begin{exam}
The positive stable law $\cs_\alpha$ with index $0< \alpha <1$ has the cumulant transform and the Mellin transform
\[
C_{\cs_\alpha}(t) = - (-t)^\alpha, \qquad t \le 0; \qquad M_{\cs_\alpha}(t) = \frac{\Gamma(1-\frac{t}{\alpha})}{\Gamma(1-t)}, \qquad -\infty < t < \alpha, 
\]
see \cite{Zol86} or \cite{S13} for further information. 
\end{exam}

\subsection{Free convolutions}

For $\mu\in\PM$ define the Cauchy transform 
\[
G_\mu(s) = \int_0^\infty \frac{1}{s-x} \,d\mu(x), \qquad s <0, 
\]
and the number 
\begin{equation}\label{eq:negative_moment}
\Aa_\mu:= \int_0^\infty x^{-1}\,d\mu(x) \in (0,\infty], 
\end{equation}
where we understand that $\Aa_\mu=\infty$ if $\mu(\{0\})>0$. 
By calculus we see that $G_\mu' <0$ on $(-\infty,0)$ and $G_\mu((-\infty,0))=(-\Aa_\mu,0)$, so that the compositional inverse $G_\mu^{-1}\colon(-\Aa_\mu,0) \to (-\infty,0)$ is defined. The $R$-transform of $\mu$ is defined by\footnote{This definition of $R$-transform is taken from \cite[Lecture 16]{NS06}, but many papers adopt the other definition $R_\mu(t) =  G_\mu^{-1}(t)-1/t.$} 
$$
R_\mu(t) = t G_\mu^{-1}(t)-1, \qquad t \in (-\Aa_\mu,0). 
$$ 
Free additive convolution is then  characterized by
\begin{equation}\label{eq:R} 
R_{\mu\boxplus \nu} = R_\mu  + R_\nu
\end{equation} 
 on the intersection of the three transforms. The formula \eqref{eq:R} actually holds on $(-\Aa_{\mu\boxplus\nu},0)$ because  $\Aa_{\mu\boxplus \nu} \le \Aa_\mu, \Aa_\nu$. This can be proved as 
$$
\Aa_{\mu \boxplus \nu} = \lim_{\epsilon\to0^+}\tau[(X+Y+\epsilon)^{-1}] \le \lim_{\epsilon\to0^+}\tau[(X+\epsilon)^{-1}] =\Aa_\mu, 
$$
where $X, Y$ are free self-adjoint elements affiliated with a finite von Neumann algebra with normal faithful tracial state $\tau$ having the distributions $\mu,\nu$, respectively. 

If $\mu\in\PM\setminus\{\delta_0\}$ then we have  
$
[sG_\mu(s)]' = -\int_0^\infty \frac{x}{(s-x)^2}d\mu(x)<0,  
$
and so we obtain by calculus
$$
R_\mu'(t) = \frac{[s G_\mu(s)]'}{G_\mu'(s)}>0, \qquad t =G_\mu(s), ~ s \in (-\infty,0), 
$$ 
and $R_\mu((-\Aa_\mu,0)) = (-1+\mu(\{0\}),0)$. For notational simplicity we sometimes extend $R_\mu$ to the homeomorphism from $[-\Aa_\mu,0]$ onto $[-1+\mu(\{0\}),0]$. The compositional inverse $R_\mu^{-1}$ can then be defined on $(-1+\mu(\{0\}),0)$ and the $S$-transform is defined by 
\[
S_\mu(u) = \frac{R_\mu^{-1}(u)}{u}, \qquad u\in(-1+\mu(\{0\}),0). 
\] 
The free multiplicative convolution of $\mu,\nu\in\PM\setminus\{\delta_0\}$ is characterized by 
\begin{equation}\label{eq:S}
S_{\mu\boxtimes \nu} = S_\mu S_\nu  
\end{equation}
 on the intersection of the domains of three functions. Actually, it holds that 
\begin{equation}\label{eq:atoms}
(\mu \boxtimes\nu)(\{0\})=\max\{\mu(\{0\}), \nu(\{0\})\}
\end{equation} 
 according to \cite[Theorem 4.1]{Bel03}, so that the common domain for \eqref{eq:S} is the interval $(-1+\max\{\mu(\{0\}), \nu(\{0\})\},0)$. 

The $S$-transform is defined in a different way in many papers. For $\mu\in\PM\setminus\{\delta_0\}$ let $\psi_\mu$ be the moment-generating function of $\mu$ defined by 
\[
\psi_\mu(t) = \int_{\R_+} \frac{tu}{1-tu} d\mu(u), \qquad t \in (-\infty,0].   
\] 
By calculus, $\psi_\mu\colon (-\infty,0) \to (-1+\mu(\{0\}),0)$ is strictly increasing and so has the compositional inverse $\psi_\mu^{-1}$ on $(-1+\mu(\{0\}),0)$. The $S$-transform of $\mu$ can alternatively be defined by 
\[
S_\mu(t) = \frac{1+t}{t}\psi_\mu^{-1}(t), \qquad t \in (-1+\mu(\{0\}),0).
\]
The reader is referred to \cite[Section 6]{BV93} for further information. 

Finally,  the free max-convolution of $\mu,\nu \in \PM$ is characterized by 
\[
(\mu\Box\hspace{-.92em}\lor \nu)([0,t]) = \max\{\mu([0,t]) + \nu([0,t]) -1, 0\}, \qquad t\ge0,
\]
see \cite[Sections 2 and 3]{BAV06} for further information. 

Sometimes the $R$-transform of a probability measure $\mu \in \PM$, defined on $(-\Aa_\mu,0)$, has a univalent analytic continuation to a larger interval; denote by $\tR_\mu$ the analytic continuation of $R_\mu$ to the maximal interval of the form $(-\Bb_\mu,0), \Bb_\mu \in [\Aa_\mu,\infty]$ on which $\tR_\mu$ is univalent. Correspondingly, we denote by $\tS_\mu$ the analytic continuation of $S_\mu$ defined on the range of $\tR_\mu$. 

\begin{exam}\label{ex:free_stable}
The positive free stable law $\fs_\alpha$ with index $\alpha \in (0,1)$, introduced in \cite[Section 7]{BV93}, has the $R$-transform and the $S$-transform (see \cite[Proposition 3.5]{AH16a})
\[
\tR_{\fs_\alpha}(t) =  - (-t)^\alpha, \qquad t \in (-\infty,0); \qquad \tS_{\fs_\alpha}(t) =  (-t)^{\frac{1-\alpha}{\alpha}}, \qquad t\in (-
\infty,0). 
\]
This means that $\Bb_{\fs_\alpha}=\infty$. On the other hand, it is known that $\fs_\alpha$ does not have an atom, so that the range of $\FC_{\fs_\alpha} = \tR_{\fs_\alpha}|_{(-\Aa_{\fs_\alpha},0)}$ should be $(-1,0)$. This implies that $\Aa_{\fs_\alpha} = 1$. See \cite{AH16a, BP99,HK14} for further information. 
\end{exam}

\begin{exam}
The Marchenko-Pastur law $\MP_\lambda$ with rate $\lambda>0$ is defined by 
\begin{equation}
\MP_\lambda(dx) = \max\{0,1-\lambda\}\delta_0(dx) + \frac{\sqrt{(x-a)(b-x)}}{2\pi x}  \,\mathbf1_{(a,b)}(x)\,dx, 
\end{equation} 
where $a=(1-\sqrt{\lambda})^2$ and $b=(1+\sqrt{\lambda})^2$. The special case $\MP_1$ is simply denoted by $\MP$ and is called the standard Marchenko-Pastur law. The $R$-transform and the $S$-transform are known to be
\[
\tR_{\MP_\lambda}(t) = \frac{\lambda t}{1-t}, \qquad t \in(-\infty,0); \qquad \tS_{\MP_\lambda}(t) = \frac{1}{\lambda+t}, \qquad t\in (-\lambda,0). 
\]
This implies that $\Bb_{\MP_\lambda}=\infty$. 
If $\lambda>1$ then $\MP_\lambda$ does not have an atom at $0$, and hence $\Aa_{\MP_\lambda}$ is determined so that $\tR_{\MP_\lambda}|_{(-\Aa_{\MP_\lambda},0)}$ has the range $(-1,0)$. This yields that $\Aa_{\MP_\lambda}= 1/(\lambda-1)$. If $\lambda \in(0,1]$ then $\Aa_{\MP_\lambda}= \infty$ by the definition \eqref{eq:negative_moment}. We can check that $R_{\MP_\lambda}$ maps $(-\infty,0)$ onto $(-\lambda,0)=(-1+\MP_\lambda(\{0\}),0)$. In this case the domains of $R_{\MP_\lambda}$ and $\tR_{\MP_\lambda}$ coincide. 
\end{exam}

\begin{exam}
Let $\WS_{m,v}$ be the Wigner semicircle distribution with mean $m$ and variance $v>0$ such that $m \ge 2\sqrt{v}$:  
\[
\WS_{m,v}(dx) = \frac{\sqrt{4v - (x-m)^2}}{2\pi v}  \,\mathbf1_{(m-2\sqrt{v}, m +2\sqrt{v})}(x)\,dx. 
\]
Its $R$-transform is given by
\[
\tR_{\WS_{m,v}}(t)=mt + vt^2, \qquad t \in (-\frac{m}{2v},0). 
\]
Since $\WS_{m,v}$ has no atom at 0, $-\Aa_{\WS_{m,v}}$ is a solution to $mt +vt^2=-1$. The correct solution makes $\FC_{\WS_{m,v}}$ univalent on $(-\Aa_{\WS_{m,v}},0)$, so that it is the larger one. Hence $\Aa_{\WS_{m,v}} = \frac{m -\sqrt{m^2-4v}}{2v}$. 

\end{exam}

\subsection{Boolean convolutions}
The $\eta$-transform, also called the boolean cumulant transform, of $\mu\in\PM$ is defined by 
\[
\eta_\mu(t) = \frac{\psi_\mu(t)}{1+\psi_\mu(t)},\qquad t \in (-\infty,0). 
\]
Boolean additive convolution $\uplus$ and boolean max-convolution $\bmax$ are then respectively characterized by  
\begin{equation}\label{eq:sum_boole}
\eta_{\mu\uplus\nu} = \eta_\mu + \eta_{\nu},\qquad  \frac1{(\mu\cup\hspace{-.87em}\lor \nu)([0,t])}=\frac{1}{\mu([0,t])} + \frac{1}{\nu([0,t])} -1,  \quad \text{for} \quad t\ge0,
\end{equation}
where $(\mu\cup\hspace{-.92em}\lor\nu)([0,t])$ is set to be $0$ when $\mu([0,t])=0$ or $\nu([0,t])=0$. The reader is referred to \cite[Section 3]{SW97} and \cite[Section 3]{VV17} for further details.

\begin{exam}
The positive boolean stable law $\bs_\alpha$ with index $\alpha \in (0,1)$, introduced in \cite[Section 3]{SW97}, has the density (see \cite[Proposition 4]{HS15}) 
\[
\frac{\sin \pi\alpha}{\pi } \cdot \frac{x^{\alpha-1}}{x^{2\alpha} + 2 x^\alpha\cos \pi\alpha +1} \mathbf1_{(0,\infty)}(x)
\]
and the $\eta$-transform and the $S$-transform (see  \cite[Proposition 3.5]{AH16a})
\[
\eta_{\bs_\alpha}(t) = - (-t)^\alpha, \qquad t \in (-\infty,0);  \qquad S_{\bs_\alpha}(t) = \left(\frac{-t}{1+t}\right)^{\frac{1-\alpha}{\alpha}}, \qquad t \in (-1,0). 
\]
\end{exam}

\subsection{Infinitely divisible distributions}



Given an associative convolution $\star$ on the set of (Borel) probability measures, let $\ID(\R,\star)$ be the set of infinitely divisible distributions on $\R$ with respect to $\star$, that is, a probability measure $\mu$ on $\R$ belongs to $\ID(\R,\star)$ if and only if for every $n\in\N$ there exists $\mu_n \in \P$, called a convolution $n$-th root of $\mu$, such that 
$$
\mu = \underbrace{\mu_n \star \mu_n \star \cdots \star \mu_n}_{\text{$n$ fold}}.
$$
 We also set 
$
\ID(\R_+,\star) 
$
to be the class of all infinitely divisible distributions on $\R_+$ such that one may take convolution $n$-th roots from $\PM$ for all $n\in\N$. The members in the class $\ID(\R_+,\boxplus)$ are called free regular distributions introduced in \cite[Section 2.1]{PS12} and investigated further in \cite{AHS13, S11}.  

Some classes are trivial such as $\ID(\R,\lor)= \ID(\R,\uplus) = \P$ and $\ID(\R_+, \bmax)=\PM$, while some are not; for example $\ID(\R,\ast), \ID(\R,\boxplus) \subsetneqq \P$ (see \cite{S13,BV93}). 
It is known that actually $\ID(\R_+,\ast) = \ID(\R,\ast)\cap \PM$, while $\ID(\R_+,\boxplus) \subsetneqq \ID(\R,\boxplus)\cap \PM$; see \cite[p.105]{PS12} and \cite[Proposition 3.1]{S11}.

The class of infinitely divisible distributions is of interest from the viewpoint of stochastic processes of independent increments and of limit theorems. The interested reader can consult \cite{Bia98, BNT02}.

\section{homomorphisms from additive convolutions to max-convolutions} \label{sec3}

\subsection{The free case}\label{sec3F}

In order to prove that the mapping $\Theta$ defined in \eqref{eq:Theta} is a homomorphism, we start by characterizing the distribution function of  $\Theta(\mu)$.

\begin{prop}\label{prop:key}
For $\mu\in\PM$, we obtain
\[
\Theta(\mu)([0,t]) = 
\begin{cases}
0, & 0 \le t < \Aa_\mu^{-1}, \\
1+R_\mu (-t^{-1}), & t \ge \Aa_\mu^{-1},  
\end{cases}  
\]
where $\Aa_\mu^{-1}$ is set to be $0$ if $\Aa_\mu=\infty$. 
\end{prop}
\begin{rem} If $\Aa_\mu=\infty$ then the interval $[0, \Aa_\mu^{-1})$ is empty, so that the formula is simply $\Theta(\mu)([0,t]) = 
1+R_\mu (-t^{-1})$ for all $t \ge 0$. If $\Aa_\mu<\infty$ then $\mu(\{0\})=0$ and hence $\Theta(\mu)([0,\Aa_\mu^{-1}])=0$. 
\end{rem}
\begin{proof} 

Note that the identity
\[
\Theta(\mu)\left( \left[ 0,\frac{1}{S_{\MP^{-1}}(u-1)S_\mu(u-1)}\right]\right)=u
\]
holds for $u\in ((\MP^{-1}\boxtimes \mu)(\{0\}),1)=(\mu(\{0\}),1)$. Using $S_\MP(u) = 1/(u+1)$ and Haagerup-Schultz's formula \cite[Proposition 3.13]{HS07}, we obtain $S_{\MP^{-1}}(u)=1/S_{\MP}(-u-1)=-u$ for $u\in (-1,0$). Therefore we have
\[
\frac{1}{S_{\MP^{-1}}(u-1)S_\mu(u-1)}=-\frac{1}{(u-1)S_\mu(u-1)}=-\frac{1}{R_\mu^{-1}(u-1)}, \qquad u\in (\mu(\{0\}),1).
\]
Substituting $u=1+R_\mu(-1/t)$, where $t \in (\Aa_\mu^{-1},\infty)$, into the above yields that 
\[
\Theta(\mu)([0,t]) = 1+R_\mu (-t^{-1}). 
\]
Taking the limit $t\downarrow \Aa_\mu^{-1}$ further implies that 
$
\Theta(\mu)([0,\Aa_\mu^{-1}]) = \mu(\{0\}).   
$
 Thus the desired formula holds for all $t\ge \Aa_\mu^{-1}$. 

For $0\le t < \Aa_\mu^{-1}$ we may assume that $\mu(\{0\})=0$; otherwise $\Aa_\mu=\infty$ and $[0,\Aa_\mu^{-1})$ would be empty. 
We already know that $\Theta(\mu)([0,\Aa_\mu^{-1}])=0$. Therefore, the non-decreasing function $t\mapsto \Theta(\mu)([0,t])$ must be identically $0$ on $[0,\Aa_\mu^{-1}]$. 
\end{proof}

\begin{thm}\label{eq:free_hom} The mapping $\Theta$ in \eqref{eq:Theta} is a homomorphism from the semigroup $(\PM, \boxplus)$ to $(\PM,\fmax)$.\end{thm}
\begin{proof} 
The goal is to show that for any $\mu,\nu\in \PM$,
\begin{align}\label{goal}
\Theta(\mu\boxplus\nu)([0,t]) = \max\{ \Theta(\mu)([0,t]) + \Theta(\nu)([0,t])-1,0\}, \qquad t\ge0. 
\end{align}

\vspace{2mm}
\noindent
{\bf Case 1:} $t\ge \Aa_{\mu\boxplus\nu}^{-1}$. By Proposition \ref{prop:key} and the general fact $\Aa_{\mu\boxplus \nu}^{-1}\ge \Aa_\mu^{-1},\Aa_\nu^{-1}$, we have
\begin{equation}\label{eq:sum}
 \Theta(\mu\boxplus\nu)([0,t]) =  \Theta(\mu)([0,t]) +  \Theta(\nu)([0,t])-1. 
\end{equation}
 The equation \eqref{goal} thus holds for all $t\ge \Aa_{\mu\boxplus\nu}^{-1}$.

We have nothing to prove anymore if $\Aa_{\mu\boxplus\nu}=\infty$, so we may assume that $\Aa_{\mu\boxplus\nu}<\infty$ below. This in particular implies that $\Theta(\mu\boxplus\nu)([0,\Aa_{\mu\boxplus\nu}^{-1}])=0$ by Proposition \ref{prop:key}.  

\vspace{2mm}
\noindent
{\bf Case 2:} $0 \le t < \Aa_{\mu\boxplus \nu}^{-1}$.  
 We have $\Theta(\mu\boxplus\nu)([0,t])=0$ by Proposition \ref{prop:key}. Since $u\mapsto\Theta(\mu)([0,u]) + \Theta(\nu)([0,u])$ is non-decreasing on $\R_+$ and $t <  \Aa_{\mu\boxplus\nu}^{-1}$, we have
\begin{align*}
\Theta(\mu)([0,t]) + \Theta(\nu)([0,t])-1
&\le \Theta(\mu)([0,\Aa_{\mu\boxplus\nu}^{-1}]) + \Theta(\nu)([0,\Aa_{\mu\boxplus\nu}^{-1}])-1 \\
&= \Theta(\mu\boxplus\nu)([0,\Aa_{\mu\boxplus\nu}^{-1}]) \\
&=0, 
\end{align*}
where \eqref{eq:sum} was used on the second line. This implies \eqref{goal}. 
\end{proof}

Note that $\MP^{-1}$ coincides with the free stable law $\fs_{1/2}$, so we may write 
\begin{equation}\label{eq:Theta2}
\Theta(\mu) = \Phi(\fs_{1/2}\boxtimes \mu), 
\end{equation}
which is to be compared with \eqref{eq:Xi2}. 

A version of Theorem \ref{eq:free_hom} holds in the setting of partially defined free convolution semigroups. Recall that each $\mu \in \PM$ associates the partially defined free convolution semigroup $\{\mu^{\boxplus t}\}_{t\ge1} \subset \PM$ such that $R_{\mu^{\boxplus t}}=tR_\mu$ on their common domain, see \cite[Corollary 14.13]{NS06} and \cite[Theorem 2.5]{BB04}.

\begin{prop}\label{prop:free_hom} Let $t\ge1$ and $\mu$ be a probability measure on $\PM$. Then 
\[
\Theta(\mu^{\boxplus t}) = \Theta(\mu)^{\fmaxxx t}. 
\]
\end{prop}
\begin{proof}
A natural proof would be to follow the lines of Theorem \ref{eq:free_hom}, but here we provide an alternative proof by exploiting several known identities:   
 \begin{align}
 \Theta(\mu^{\boxplus t}) 
 &= \Phi(\fs_{1/2} \boxtimes (\mu^{\boxplus t})) \notag \\ 
 &= \Phi(D_{t^{-2}} ((\fs_{1/2}^{\boxplus t}) \boxtimes (\mu^{\boxplus t}))) \label{eq:stability}\\ 
 &=  \Phi(D_{t^{-1}}(\fs_{1/2} \boxtimes \mu)^{\boxplus t}) \label{eq:plus-times}\\
 &=  \Phi(\fs_{1/2} \boxtimes \mu)^{\fmaxxx t} \label{eq:ueda} \\
 &= \Theta(\mu)^{\fmaxxx t}, \notag
 \end{align}
where the stability condition $\fs_{1/2}^{\boxplus t} = D_{t^2}(\fs_{1/2})$, the distributive relation $(\mu \boxtimes \nu)^{\boxplus t} = D_{1/t}((\mu^{\boxplus t}) \boxtimes (\nu^{\boxplus t}))$ \cite[Proposition 3.5]{BN08} and \eqref{eq:Ueda} are used in the steps \eqref{eq:stability}, \eqref{eq:plus-times} and  \eqref{eq:ueda}, respectively. 
\end{proof}
The proof above shows how the homomorphism property of $\Theta$ is related to the first identity in \eqref{eq:Ueda}.

\begin{exam} Let $\alpha \in(0,1)$. Recall from Example \ref{ex:free_stable} that $R_{\fs_\alpha}(t) = -(-t)^\alpha$ on $(-1,0)$ with $\Aa_{\fs_\alpha}=1$. Then Proposition \ref{prop:key} yields that 
\[
\Theta(\fs_\alpha)([0,t])=\Phi(\MP^{-1} \boxtimes \fs_\alpha)([0,t]) = (1-t^{-\alpha})_+, \qquad t \ge0,  
\]
so that $\Theta(\fs_\alpha)$ is the Pareto distribution, which is a free extreme value distribution \cite[Definition 6.7]{BAV06}. It also coincides with $\uni^{-\frac{1}{\alpha}}$, where $\uni$ is the uniform distribution on $(0,1)$ as observed in \cite[Remark 6]{HSW20}.  
\end{exam}

\subsection{The boolean case} \label{sec3B}


In order to prove that the mapping $\Xi$ defined in \eqref{eq:Xi} is a homomorphism, we start by characterizing the distribution function of  $\Xi(\mu)$.

\begin{lem}\label{lem:atom}
Let $\mu \in \PM$. Then $\Aa_{\MP\boxtimes\mu}=\infty$. 
\end{lem} 
\begin{proof} Recall that $\tR_{\MP\boxtimes\mu}(t) = \psi_\mu(t)$ for $t \in (-\infty,0)$. If $\mu=\delta_0$ then the conclusion is obvious, so assume that $\mu \neq \delta_0$. Then $\tR_{\MP\boxtimes\mu}$ bijectively maps $(-\infty,0)$ onto $(-1+\mu(\{0\}),0) = (-1+(\MP\boxtimes\mu)(\{0\}),0)$, where \eqref{eq:atoms} was used. Since the last interval is the range of $\FC_{\MP\boxtimes\mu}$, we conclude that $\Aa_{\MP\boxtimes\mu}=\infty$. 
\end{proof}
\begin{prop}
\label{prop:key2}
Consider $\mu\in \PM$. Then 
\[
\Xi(\mu)([0,t]) = 1 + \psi_\mu(-t^{-1}), \qquad t \ge 0,
\]
where $\psi_\mu(-0^{-1}):=\lim_{t\to-\infty}\psi_\mu(t)=-1+\mu(\{0\})$.
\end{prop}
\begin{proof} 
It suffices to replace $\mu$ with $\MP\boxtimes \mu$ in Proposition \ref{prop:key} and use the identity \cite[Proposition 12.18]{NS06}
\begin{equation}\label{eq:CFP}
\tR_{\MP\boxtimes \mu}(t) = \psi_\mu(t),  
\end{equation}
together with Lemma \ref{lem:atom} which entails that $\tR_{\MP\boxtimes \mu}$ and $\FC_{\MP\boxtimes \mu}$ have the same domain $(-\infty,0)$. 
\end{proof}

\begin{thm}\label{thm:boole_hom} The mapping $\Xi$ in \eqref{eq:Xi} is a homomorphism from the semigroup $(\PM, \uplus)$ to $(\PM,\bmax)$.\end{thm}
\begin{proof}
The goal is to demonstrate that 
\[
\frac{1}{ \Xi(\mu \uplus \nu)([0,t])} =   \frac1{\Xi(\mu)([0,t])} + \frac1{\Xi(\nu)([0,t])} -1.  
\]
Note that by Proposition \ref{prop:key2} 
\[
 \frac1{\Xi(\mu)([0,t])} = \frac1{1 + \psi_\mu(-t^{-1})} = 1-\eta_\mu(-t^{-1}). 
\]
The goal is then achieved from the additivity relation \eqref{eq:sum_boole}. 
\end{proof}

Note that the measure $\MP\boxtimes\MP^{-1}$ is actually the boolean stable law $\bs_{1/2}$, so that we may write 
\begin{equation}\label{eq:Xi2}
\Xi(\mu) = \Phi(\bs_{1/2}\boxtimes \mu), 
\end{equation}
which shows a resemblance with \eqref{eq:Theta2}. Also, the obvious formula 
\begin{equation}\label{eq:Xi3}
\Xi(\mu)=\Theta(\MP\boxtimes\mu)
\end{equation}
relates $\Xi$ and $\Theta$. 

Actually, the mapping $\Xi$ was already discussed in \cite[Theorem 5.4]{AH16a}, where the following formula was obtained. 
\begin{prop} For $\mu \in \PM$ we have 
\[
\Xi(\mu) = \ed \circledast \ed^{-1} \circledast \mu,   
\]
where $\ed$ is the exponential distribution with density $e^{-x} \mathbf1_{(0,\infty)}(x)$. 
\end{prop}
\begin{rem}
The distribution $ \ed \circledast \ed^{-1}$ is the Pareto distribution having the density $(1+x)^{-2}\mathbf1_{(0,\infty)}(x)$.  
\end{rem}

Theorem \ref{thm:boole_hom} has a version for convolution semigroups. Recall that each $\mu \in \P$ associates the $\uplus$-convolution semigroup $\{\mu^{\uplus t}\}_{t\ge0}\subset\P$ such that $\eta_{\mu^{\uplus t}} = t \eta_\mu$ \cite[Proposition 3.1]{SW97}. It follows from \cite[Theorem 6.2]{Has10} that actually $\{\mu^{\uplus t}\}_{t\ge0}\subset\PM$ whenever $\mu \in \PM$.

\begin{prop} Let $t\ge0$ and $\mu\in \PM$. Then 
\[
\Xi(\mu^{\uplus t}) = \Xi(\mu)^{\bmaxxx t}. 
\]
\end{prop}
\begin{proof}
The proof is similar to Proposition \ref{prop:free_hom}  or Theorem \ref{thm:boole_hom} and is omitted. 
\end{proof}

\begin{exam} Let $\alpha \in(0,1)$. Since $\eta_{\bs_\alpha}(t) = -(-t)^\alpha$, Proposition \ref{prop:key} yields that 
\[
\Xi(\bs_\alpha)([0,t])=\Phi(\MP^{-1} \boxtimes\MP \boxtimes \bs_\alpha)([0,t]) = \frac{t^\alpha}{1+t^\alpha}, \qquad t \ge0,  
\]
so that $\Xi(\bs_\alpha)$ is the Dagum distribution, which is a boolean extreme value distribution \cite[Corollary 4.1]{VV17}. It coincides with $(\ed^{-1} \circledast \ed)^{1/\alpha}$ as observed in \cite[Remark 6]{HSW20}, which also appears in Section \ref{sec4S}; see \eqref{eq:b1}. 

\end{exam}

\subsection{The classical case}\label{sec3C}

We define a classical analogue of the mappings $\Theta$ and $\Xi$.  In the following we understand that $\log 0=-\infty$. 
We begin with the following lemma which ensures that the definition makes sense.


\begin{lem}\label{lem:Gamma} For every $\mu\in\PM$, the mapping
\[
(0,\infty) \ni t\mapsto \int_0^\infty C_\mu(-x/t)e^{-x}\,dx \in (\log \mu(\{0\}),0)
\]
is a strictly increasing homeomorphism.  We will extend it to a homeomorphism from $[0,\infty)$ onto $[\log \mu(\{0\}),0)$ when convenient. 
\end{lem}
\begin{proof}
Let $f(u):= -C_\mu(-u)$. It is clear that $f$ is a non-negative strictly increasing continuous function on $[0,\infty)$ such that $f(0)=0$. By the dominated convergence theorem we get $
\lim_{t\to-\infty} \int_{\R_+} e^{t x} \,d\mu(x) = \mu(\{0\}) $ and hence $\lim_{u\to\infty}f(u) = -\log \mu(\{0\})$.  

We find constants $a, b\ge0$ depending only on $\mu$ such that $f(u) \le a u +b$. This is carried out as follows. Choose $a \ge 0$ such that 
$\mu([0,a])>0$. Then  
\begin{align*} 
C_\mu(-u) &= \log \int_{\R_+} e^{-ux}\,d\mu(x)  \ge \log \int_{[0,a]} e^{-ux}\,d\mu(x) \\
&\ge \log \int_{[0,a]} e^{-a u}\,d\mu(x) = -a u + \log \mu([0,a]),  
\end{align*}
so we can set $b = -  \log \mu([0,a])$ to get the inequality $f(u) \le a u +b$. 
This implies that  
$x\mapsto f(x/t)e^{-x} \in L^1((0,\infty),dx)$ for every $t \in (0,\infty)$. 

It is obvious that 
\[
t \mapsto \int_0^\infty f(x/t)e^{-x}\,dx
\]
is strictly decreasing, and also continuous by the dominated convergence theorem.  
The monotone convergence theorem implies that 
\[
\lim_{t\to0^+} \int_0^\infty f(x/t)e^{-x}\,dx =-\log \mu(\{0\}), 
\]
and since $|f(x/t)| \le f(x)$ for all $x \in (0,\infty), t\in (1,\infty)$, the dominated convergence theorem implies that 
\[
\lim_{t\to\infty} \int_0^\infty f(x/t)e^{-x}\,dx =0. 
\]
\end{proof}

According to Lemma \ref{lem:Gamma}, for each $\mu \in \PM$ we are able to define a probability measure $\Omega(\mu) \in \PM$ by 
\[
\Omega(\mu)([0,t]) =\exp\left(\int_0^\infty C_\mu(-x/t)e^{-x}\,dx\right), \qquad t \in (0,\infty) \qquad \text{and} \qquad \Omega(\mu)(\{0\})=\mu(\{0\}). 
\]
This definition is given so that the intertwining relations in Theorem \ref{thm:inter} below hold. 

\begin{thm} \label{thm:classical_hom}
The mapping $\Omega \colon (\PM, \ast) \to (\PM, \lor)$ is a homomorphism.  
\end{thm}
\begin{proof} The following straightforward computation is valid for $t>0$: 
\begin{align*}
\Omega(\mu \ast \nu)([0,t]) &=\exp\left(\int_0^\infty C_{\mu\ast\nu}(-x/t)e^{-x}\,dx\right) \\
&= \exp\left(\int_0^\infty C_{\mu}(-x/t)e^{-x}\,dx +\int_0^\infty C_{\nu}(-x/t)e^{-x}\,dx \right)\\
 &= \Omega(\mu) ([0,t])  \Omega(\nu) ([0,t]) \\
 &= (\Omega (\mu) \lor \Omega(\nu)) ([0,t]).  
\end{align*}
It also holds for $t=0$ by the right-continuity or direct computations.  
\end{proof}

We introduce another mapping $\tilde\Omega$ by 
\[
\tilde\Omega(\mu) = \ed^{-1}\circledast \mu. 
\] 
A straightforward computation yields 
\[
\tilde\Omega(\mu)([0,t]) = \int_{(0,\infty)} e^{- x/t} \,d\mu(x) = \exp(C_\mu(-t^{-1})), \qquad t \in (0,\infty). 
\]
We can easily show that the mapping $\tilde\Omega \colon (\PM, \ast) \to (\PM, \lor)$ is a homomorphism.  
This mapping will be studied in Section \ref{sec4S}.

\begin{exam}  By straightforward computations we get 
$$
\Omega(\cs_\alpha)([0,t]) = \exp (-\Gamma(1+\alpha)t^{-\alpha}), \qquad \tilde\Omega(\cs_\alpha)([0,t]) = \exp (-t^{-\alpha}),\qquad t \in (0,\infty),   
$$
so that both $\Omega(\cs_\alpha)$ and $\tilde\Omega(\cs_\alpha)$ are Fr\'etchet distributions. Note that $\tilde\Omega(\cs_\alpha)$ coincides with $\ed^{-1/\alpha}$. 
\end{exam}

\section{Intertwining relations}\label{sec4}
We will observe several intertwining relations for the homomorphisms $\Theta, \Xi, \Omega, \tilde\Omega $ together with Bercovici--Pata bijections and some homomorphisms related with stable distributions. In what follows when we talk about a homeomorphism between sets of probability measures, the continuity is always concerned with the weak convergence of probability measures. 

\subsection{Relations with Bercovici-Pata bijections}

Let $\Lambda$ be the Bercovici-Pata bijection (see \cite[Theorem 1.2]{BP99} and \cite[Definition 3.1]{BNT02}) which is a homeomorphic homomorphism from the semigroup $(\ID(\R,\ast),\ast)$ to $(\ID(\R,\boxplus),\boxplus)$. 
Its restriction $\Lambda|_{\ID(\R_+,\ast)}$ is a bijection from $\ID(\R_+,\ast)$ onto $\ID(\R_+,\boxplus)$.  According to \cite[Theorem 4.1]{BNT04}, the measure $\Lambda(\mu)$ for  $\mu\in \ID(\R_+,\ast)$ is characterized by 
\begin{equation}\label{eq:BP}
\tR_{\Lambda(\mu)}(t) = \int_0^\infty C_\mu(t x)e^{-x}\,dx,   \qquad t \in (-\infty,0).   
\end{equation}
Let $\Lambda^\lor$ be the homomorphism from the semigroup $(\PM, \lor)$ to $(\PM, \fmax)$:
\[
\Lambda^\lor (\mu)([0,t]) = \max\{1+\log \mu([0,t]),0\},  \qquad t\ge0,
\]
introduced in \cite[Section 6]{BAV06}, where the reader may understand $\Lambda^\lor (\mu)([0,t]) =0$ if $\mu([0,t])=0$.

A homeomorphic homomorphism $\mathcal{X}\colon (\mathcal{P}(\R),\uplus) \to(\ID(\R,\ast),\ast)$ is defined by 
\[
\eta_{\mathcal{X}^{-1}(\mu)} (t) = \int_0^\infty C_\mu(t x)e^{-x}\,dx = \tR_{\Lambda(\mu)}(t), \qquad t \in(-\infty,0),  
\] 
whose restriction to $\PM$ is a bijection onto $\ID(\R_+,\ast)$. Notice that $\mathcal{X}=\Lambda^{-1}\circ\Lambda_{bf}$ on $\PM$.  
We also have $\Omega(\mathcal{X}(\mu))([0,t])=\exp(\eta_\mu (-1/t))$ for $\mu\in \PM$, which implies in the limit $t\to0^+$ that 
\begin{align*}
\mathcal{X}(\mu)(\{0\})=\begin{cases}
\exp(1-\mu(\{0\})^{-1}), & \text{ if } \mu(\{0\})>0, \\
0, & \text{ if } \mu(\{0\})=0. 
\end{cases}
\end{align*}
Furthermore, let $\mathcal{X}^\lor$ be the homeomorphic homomorphism from the semigroup $(\PM,\cup \hspace{-.65em}\lor )$ to $(\PM, \lor)$:
\[
\mathcal{X}^\lor(\mu)([0,t])=\exp(1-\mu([0,t])^{-1}), \qquad t\ge 0,
\]
introduced in \cite[Section 4]{VV17}, where we understand $\mathcal{X}^\lor(\mu)([0,t])=0$ if $\mu([0,t])=0$. 

As a preliminary, the point mass of measures at $0$ is investigated below.

\begin{lem}\label{lem:atom_BP}
Let $\mu \in \ID(\R_+,\ast)$. 
\begin{enumerate}[label=\rm(\arabic*)]
\item\label{item:BP2} If $\mu(\{0\})<1/e$ then $\Aa_{\Lambda(\mu)}<\infty$. 
\item\label{item:BP3} If $\mu(\{0\})\ge 1/e$ then $\Aa_{\Lambda(\mu)}=\infty$. 
\item \label{item:BP4} $\Lambda(\mu)(\{0\}) = \max\{1+\log \mu(\{0\}),0\}$. 
\end{enumerate}
\end{lem} 
\begin{proof} First note that, by Lemma \ref{lem:Gamma} and \eqref{eq:BP}, $\tR_{\Lambda(\mu)}$ maps $(-\infty,0)$ onto $(\log\mu(\{0\}),0)$. 

Suppose that $\mu(\{0\})<1/e$. Since $\log\mu(\{0\}) <-1$, we have $\Aa_{\Lambda(\mu)}<\infty$; otherwise the range of $\FC_{\Lambda(\mu)}$ would contain $-1$.  The established fact $\Aa_{\Lambda(\mu)}<\infty$ implies that $\Lambda(\mu)(\{0\})=0$. 

Suppose that $\mu(\{0\})\ge1/e$, which means $\log\mu(\{0\}) \ge -1$.  If  $\Aa_{\Lambda(\mu)}<\infty$ then there would exist $\epsilon\in(0,1)$ such that $R_{\Lambda(\mu)}((-\Aa_{\Lambda(\mu)},0)) = (-1+\epsilon,0)$. This implies that $\Lambda(\mu)(\{0\})=\epsilon>0$, a contradiction to $\Aa_{\Lambda(\mu)}<\infty$. Therefore, we must have $\Aa_{\Lambda(\mu)}=\infty$. Moreover, in this case we have $\FC_{\Lambda(\mu)}((-\Aa_{\Lambda(\mu)},0)) = (\log\mu(\{0\}),0)$, so that $-1+\Lambda(\mu)(\{0\})=\log\mu(\{0\}).$
\end{proof}


\begin{exam} As a special case of Lemma \ref{lem:atom_BP} \ref{item:BP4}, the Poisson distribution with rate $\lambda >0$ has the mass $e^{-\lambda}$ at $0$, while the Marchenko-Pastur distribution with rate $\lambda$, which is the image of the Poisson distribution by $\Lambda$, has the mass $\max \{1-\lambda,0\}$ at 0. 
\end{exam}

\begin{thm}\label{thm:inter} The following intertwining relations  
\[
\Lambda^\lor \circ \Omega = \Theta \circ \Lambda \quad  \text{on} \quad \ID(\R_+,\ast) \qquad\text{and}\qquad  \Omega\circ \mathcal{X} = \mathcal{X}^\lor\circ \Xi \quad \text{on}\quad  \PM 
\]
hold, which yield the following commuting diagram for homomorphisms: 
\[
\begin{diagram}
\node{(\PM,\uplus)}  \arrow{e,t}{\mathcal{X}} \arrow{s,l}{\Xi} \node{(\ID(\R_+,\ast),\ast)}\arrow{e,t}{\Lambda} \arrow{s,l}{\Omega}\node{(\ID(\R_+,\boxplus),\boxplus)} \arrow{s,r}{\Theta}\\
\node{(\PM,\cup\hspace{-.68em}\lor)} \arrow{e,b}{\mathcal{X}^\lor} \node{(\PM,\lor)}  \arrow{e,b}{\Lambda^\lor} \node{(\PM,\Box \hspace{-.75em}\lor).}
 \end{diagram}
\]
\end{thm}
\begin{proof}
For $\mu \in \ID(\R_+,\ast)$ we have 
\[
(\Theta \circ \Lambda(\mu))([0,t]) = \begin{cases}
0, & 0 \le t < \Aa_{\Lambda(\mu)}^{-1}, \\
1+\int_0^\infty C_\mu(-x/t)e^{-x}\,dx, & t \ge \Aa_{\Lambda(\mu)}^{-1},
\end{cases}  
\]
while
\begin{align*}
(\Lambda^\lor \circ \Omega(\mu))([0,t])&=\begin{cases}
0, & 0 \le \Omega(\mu)([0,t]) < e^{-1}, \\
1+\int_0^\infty C_\mu(-x/t)e^{-x} dx, & \Omega(\mu)([0,t]) \ge e^{-1}. 
\end{cases}
\end{align*}
Hence it suffices to show that $\{t \in [0,\infty):  \Omega(\mu)([0,t])\ge e^{-1}\} = [\Aa_{\Lambda(\mu)}^{-1},\infty)$. 

\vspace{2mm}
\noindent
{\bf Case 1.} If $\mu(\{0\})\ge 1/e$ then $\Aa_{\Lambda(\mu)}^{-1}=0$ by Lemma \ref{lem:atom_BP}, and $\Omega(\mu)([0,t]) \ge \mu(\{0\})\ge 1/e$ for all $t\ge0$. 

\vspace{2mm}
\noindent
{\bf Case 2.} If $\mu(\{0\})< 1/e$ then $\Aa_{\Lambda(\mu)}^{-1}>0$ and $\Lambda(\mu)(\{0\})=0$ by Lemma \ref{lem:atom_BP}. Then the condition 
$
 \Omega(\mu)([0,t])\ge e^{-1}
$
is equivalent to that $\tR_{\Lambda(\mu)}(-1/t)=\int_0^\infty C_\mu(-x/t)e^{-x}\,dx\ge -1$. Since $\Lambda(\mu)(\{0\})=0$, this is further equivalent to that $-1/t \ge  -\Aa_{\Lambda(\mu)}$. 

\vspace{2mm}
\noindent
Thus the first intertwining relation has been established. 

For the second intertwining relation, given $\mu\in \PM$ and $t\in (0,\infty)$, we proceed as
\begin{align*}
(\Omega\circ\mathcal{X}(\mu))([0,t])&=\exp(\eta_\mu(-1/t))\\
&=\exp\left(1-\frac{1}{\Xi(\mu)([0,t])} \right)\\
&=(\mathcal{X}^\lor \circ\Xi(\mu))([0,t]).
\end{align*}
By the right-continuity, this holds for $t=0$ as well. 
\end{proof}

An obvious corollary of Theorem \ref{thm:inter} is the intertwining relation between boolean and free convolutions, which is described explicitly below. The homeomorphic homomorphism $\Lambda_{bf}:=\Lambda\circ \mathcal{X}\colon (\P,\uplus) \to (\ID(\R,\boxplus),\boxplus)$ is called the boolean-to-free Bercovici--Pata bijection (denoted by $\mathbb B$ in \cite[(2.20)]{BN08}). Its restriction $\Lambda_{bf}|_{\PM}$ induces a bijection from $\PM$ onto $\ID(\R_+,\boxplus)$. Let $\Lambda^\lor_{bf}:=\Lambda^\lor\circ \mathcal{X}^\lor \colon(\PM,\bmax) \to (\ID(\R_+, \fmax),\fmax)$ be the max-boolean-to-free homomorphism (denoted by $\mathbf{B}_1^M$ in \cite[Proposition 5.5]{Ua19}).  They are characterized by 
\begin{equation}\label{R-eta}
\tR_{\Lambda_{bf}(\mu)} = \eta_\mu \quad \text{on} \quad (-\infty,0) \qquad \text{and} \qquad \Lambda^\lor_{bf}(\mu)([0,t]) = \max\left\{2 - \frac{1}{\mu([0,t])},0\right\} \quad \text{for} \quad t\ge0, 
\end{equation}
where the reader may understand $\Lambda_{bf}^\lor (\mu)([0,t])=0$ if $\mu([0,t])=0$. From the latter formula, it can be observed that $\Lambda^\lor_{bf}$ is not injective.

It follows from the commuting diagram for homomorphisms in Theorem \ref{thm:inter} that 
\begin{equation}\label{eq:interBF}
\Theta\circ\Lambda_{bf}=\Lambda_{bf}^\lor\circ \Xi\qquad \text{on}\qquad \PM, 
\end{equation}
giving a relation between $\Theta$ and $\Xi$ different from \eqref{eq:Xi3}. The formula \eqref{eq:interBF} can also be more directly proved, without resorting to Theorem \ref{thm:inter} but with similar arguments, based on the following facts for $\mu\in \PM$ which may be of independent interest.  
\begin{enumerate}[label=\rm(\arabic*)]
\item\label{item2} If $\mu(\{0\})<1/2$ then $\Aa_{\Lambda_{bf}(\mu)}<\infty$. 
\item\label{item3} If $\mu(\{0\})\ge 1/2$ then $\Aa_{\Lambda_{bf}(\mu)}=\infty$. 
\item \label{item4} $\Lambda_{bf}(\mu)(\{0\}) = \max\{2-\frac1{\mu(\{0\})},0\}$. 
\end{enumerate}
The proofs of those facts are similar to Lemma \ref{lem:atom_BP}. The details are omitted here. 
 

\subsection{Relations with stable distributions}\label{sec4S}

For the sake of notational convenience, for $c \in \mathbb R $ and $\nu\in\PM$ let $\PO_c(\nu)$ stand for $\nu^c$, the push-forward of $\nu$ by the map $t\mapsto t^{c}$.  
We can easily check that $\PO_c$ is a homomorphism on any of the semigroups $(\PM,\lor)$, $(\PM, \fmax)$ and $(\PM,\bmax)$. Therefore, the mappings 
\begin{equation}\label{eq:composition1}
\PO_c\circ \tilde\Omega, \qquad \PO_c\circ  \Theta, \quad \text{and}\quad  \PO_c \circ \Xi 
\end{equation}
are homomorphisms from the corresponding additive convolution to the max one. 

On the other hand, in \cite[Theorem 4.14]{AH16a} the following mappings
\[
\CH_\alpha(\nu) = \cs_\alpha \circledast \nu^{1/\alpha}, \qquad \FH_\alpha(\nu) = \fs_\alpha \boxtimes \nu^{\boxtimes1/\alpha}, \qquad\text{and}\qquad \BH_\alpha(\nu) = \bs_\alpha \circledast \nu^{1/\alpha}
\] 
have been introduced and shown to be homomorphisms on $(\PM,\ast)$, $(\PM,\boxplus)$ and $(\PM,\uplus)$, respectively. Moreover, according to \cite[Theorem 4.5]{AH16a}  the identity
\[
 \bs_\alpha \circledast \nu^{1/\alpha} =  \bs_\alpha \boxtimes \nu^{\boxtimes1/\alpha}
\]
holds, which provides an alternative formula for $\BH_\alpha(\nu)$. 

The mappings 
\begin{equation}\label{eq:composition2}
 \tilde\Omega \circ \CH_\alpha, \qquad \Theta \circ \FH_\alpha, \quad \text{and}\quad  \Xi \circ \BH_\alpha
\end{equation}
are therefore homomorphisms from the corresponding additive convolution to the max one. We will relate \eqref{eq:composition1} and \eqref{eq:composition2} below. 

\begin{lem} Let $\alpha \in (0,1)$. The following identities hold:
\begin{align}
\ed^{-1} \circledast \cs_\alpha\ &= \ed^{-1/\alpha}, \label{eq:c1} \\
\MP^{-1} \boxtimes \fs_\alpha &= (\MP^{-1})^{\boxtimes 1/\alpha}, \label{eq:f1} \\
(\ed \circledast \ed^{-1} )\circledast \bs_\alpha &= (\ed \circledast \ed^{-1} )^{1/\alpha}, \label{eq:b1}\\
(\MP \boxtimes \MP^{-1}) \boxtimes \bs_\alpha &= (\MP \boxtimes \MP^{-1})^{\boxtimes 1/\alpha}.  \label{eq:b2}
\end{align}
\end{lem}
\begin{proof}
The first identity follows from the Mellin transforms
$M_{\cs_\alpha}(t) = \Gamma(1-\frac{t}{\alpha})/\Gamma(1-t)$ and $M_\ed(t) = \Gamma(1+t).$  
The third identity follows from the first one and $\bs_\alpha =\cs_\alpha \circledast \cs_\alpha^{-1}$ \cite[Proposition 4.12]{AH16a}. The second identity follows from the fact $\MP^{-1}= \fs_{1/2}$ and the $S$-transform  
$
S_{\fs_\alpha}(t) = (-t)^{\frac{1-\alpha}{\alpha}}. 
$ 
The fourth identity follows from the second one and $\bs_\alpha  = \fs_\alpha \boxtimes \fs_\alpha^{-1}$ \cite[Proposition 4.12]{AH16a}. 
\end{proof}

\begin{rem} 
The identities above somehow suggest that $\MP$ corresponds to $\ed$. This correspondence has been observed in other contexts in free probability, see e.g.\ \cite{Sz15} and \cite[Remark 4.5]{HSz19}.
\end{rem}

\begin{prop} For every $\alpha \in (0,1)$ we have  
\[
\tilde\Omega \circ \CH_\alpha  = \PO_{1/\alpha} \circ \tilde\Omega, \qquad \Theta\circ \FH_\alpha = \PO_{1/\alpha} \circ \Theta, \qquad \Xi\circ\BH_\alpha = \PO_{1/\alpha} \circ \Xi.   
\]  
\end{prop}

\begin{proof}
The first identity is a simple combination of \eqref{eq:c1} and $\tilde\Psi(\nu) = \ed^{-1}\circledast \nu$. The second identity follows from 
\[
\MP^{-1} \boxtimes \fs_\alpha \boxtimes \nu^{\boxtimes 1/\alpha} = (\MP^{-1}\boxtimes\nu)^{\boxtimes 1/\alpha}, 
\]
which is a consequence of \eqref{eq:f1}, and the fact that $\Phi (\mu^{\boxtimes t})=\Phi(\mu)^t$ for all $t\ge1$. The third identity follows from \eqref{eq:b1} and the formula 
$\Xi(\mu) = \ed \circledast \ed^{-1} \circledast \mu$; note that one can alternatively use \eqref{eq:b2} and the formula $\Xi(\mu) = \Phi(\MP \boxtimes \MP^{-1}\boxtimes \mu)$. 
 \end{proof}
 
 \begin{rem} There seems no concise formula that relates $\Omega \circ \CH_\alpha$ with $\PO_{1/\alpha} \circ \Omega$. 
 \end{rem}

 \section*{Acknowledgements} 
 
T.H.\ is supported by JSPS Grant-in-Aid for Young Scientists 19K14546. This research is an outcome of Joint Seminar supported by JSPS and CNRS under the Japan-France Research Cooperative Program.


\end{document}